\documentclass{article}
\usepackage[a4paper, total={7in, 10in}]{geometry}
\usepackage[T1]{fontenc}
\usepackage{graphicx}
\usepackage{mathtools}
\usepackage{amsmath,amsthm,amssymb,mathrsfs,amstext, titlesec,enumitem, comment, graphicx, color, xcolor, stmaryrd,mathabx}
\usepackage{thmtools}
\usepackage{xpatch}
\usepackage{tikz-cd}
\usepackage{nameref}
\usepackage{hyperref}
\makeatletter
\def\Ddots{\mathinner{\mkern1mu\raise\p@
\vbox{\kern7\p@\hbox{.}}\mkern2mu
\raise4\p@\hbox{.}\mkern2mu\raise7\p@\hbox{.}\mkern1mu}}
\makeatother
\newtheorem{theorem}{Theorem}[section]

\theoremstyle{definition}
\newtheorem{definition}[theorem]{Definition}
\newtheorem{remark}[theorem]{Remark}

\begin{document}
	
	\title{The finite products of shifted primes and Moreira's Theorem}

	\date{}
	\author{Pintu Debnath
		\footnote{Department of Mathematics,
			Basirhat College,
			Basirhat-743412, North 24th parganas, West Bengal, India.\hfill\break
			{\tt pintumath1989@gmail.com}}
	}
	\maketitle

\begin{abstract}
Let $r\in\mathbb{N}$ and $\mathbb{N}=\bigcup_{i=1}^{r}C_{i}$. Do there exist $x,y\in\mathbb{N}$ and $i\in\left\{1,2,\ldots,r\right\}$ such that $\left\{x,y,xy,x+y\right\}\subseteq C_{i}$? This is still an unanswered question asked by N. Hindman. Joel Moreira in [Annals of Mathematics 185 (2017) 1069-1090] established a partial answer to this question and proved that for infinitely many $x,y\in\mathbb{N}$, $\left\{x,xy,x+y\right\}\subseteq C_{i}$ for some $i\in\left\{1,2,\ldots,r\right\}$, which is called Moreira's Theorem. Recently,  H. Hindman and D. Strauss established a refinement of Moreira's Theorem and proved that for infinitely many $y$, $\left\{x\in\mathbb{N}:\left\{x,xy,x+y\right\}\subseteq C_{i}\right\}$ is a piecewise syndetic set. In this article, we will prove infinitely many $y\in FP\left(\mathbb{P}-1\right)$  such that $\left\{x\in\mathbb{N}:\left\{xy,x+f(y):f\in F\right\}\subseteq C_{i}\right\}$ is piecewise syndetic, where $F$ is a finite subset of $x\mathbb{Z}\left[x\right]$. We denote $\mathbb{P}$ is the set of prime numbers in $\mathbb{N}$ and $FP\left(\mathbb{P}-1\right)$ is the set of all finite products of distinct elements of  $\mathbb{P}-1$.

\end{abstract}

\textbf{Keywords:}  Piecewise syndetic set, $IP_{r}$-set, Moreira's Theorem, Polynomial van der Waerden’s Theorem.

	\textbf{MSC 2020:} 05D10, 22A15, 54D35

\section{Introduction}

We start this introductory section with the statement of  Moreira's Theorem:

\begin{theorem}\label{Moreira's Theorem}\cite[Corollary 1.5]{M17}
    For any finite coloring of $\mathbb{N}$ there exist infinitely many $x,y\in\mathbb{N}$ such that $\left\{x,xy,x+y\right\}$ is monochromatic.
\end{theorem}

\begin{definition}\textbf{(Piecewise syndetic)}
   Let $\left(S,+\right)$ be a commutative semigroup and let $A\subseteq S$.  
 A is piecewise syndetic if and only if there exists $G\in\mathcal{P}_{f}\left(S\right)$
such that for every $F\in\mathcal{P}_{f}\left(S\right)$, there is
some $x\in S$ such that $F+x\subseteq\cup_{t\in G}(-t+A)$. Here $\mathcal{P}_{f}\left(S\right)$ is the  set of all finite subsets of $S$.
\end{definition}

In \cite[Corollary 1.11]{HS24}, N. Hindman and D. Strauss proved the following refinement of Moreira's Theorem:
\begin{theorem}\label{Moreira By HS}
     Let $r\in\mathbb{N}$ and let $\mathbb{N}=\bigcup_{i=1}^{r}C_{i}$. There exist $i\in\left\{1,2,\ldots,r\right\}$ and infinitely many $y$ such that $\left\{x\in\mathbb{N}:\left\{x,xy,x+y\right\}\subseteq C_{i}\right\}$ is piecewise syndetic.
\end{theorem}

J. Moreira proved his Theorem in more generalized setting , one of its particular is the following:

\begin{theorem}\label{Moreira,s polynomial}\cite[Corollary 6.1]{M17}
    Let $k\in\mathbb{N}$ and $f_{1},f_{2},\ldots,f_{k}\in\mathbb{Z}\left[x\right]$ satisfy $f_{l}\left(0\right)=0$ for each $l$. Then for any finite coloring of $\mathbb{N}$, there exists $x,y\in\mathbb{N}$ such that $\left\{xy,x+f_{1}\left(y\right),\ldots,x+f_{k}\left(y\right)\right\}$ is monochromatic.
\end{theorem}

Naturally,  a question arises in our minds as to whether we can refine Theorem \ref{Moreira,s polynomial} analog of the Theorem \ref{Moreira By HS} by N. Hindman and D. Strauss. In \textbf{ Section 2}, we will provide an affirmative answer to this question. We also prove something more associated with the set of prime numbers, which is reflected in the title of this article.

Let $\mathbb{P}$ be the set of prime numbers and $\mathbb{P}-1=\left\{p-1:p\in\mathbb{P}\right\}$ and similarly $\mathbb{P}+1=\left\{p+1:p\in\mathbb{P}\right\}$. We state the following theorem from \cite{BLZ} by V. Bergelson, A. Leibman and T. Ziegler, which motivated us for this article. 

\begin{theorem}\label{Prime shift PVW}\cite[Theorem 5]{BLZ}
    For any partition $\mathbb{Z}^{d}=\bigcup_{s=1}^{c}C_{s}$ at least one of the sets $C_{s}$ has the property that for any finite set of polynomials $\vec{f}_{i}:\mathbb{Z}\rightarrow \mathbb{Z}^{d}$, $i=1,\ldots,k$, with $\vec{f}_{i}\left(0\right)=0$ for all $i$, $$\left\{n\in\mathbb{N}:\vec{a},\vec{a}+\vec{f}_{1}\left(n\right),\ldots, \vec{a}+\vec{f}_{k}\left(n\right)\in C_{s}\text{ for some } \vec{a}\in\mathbb{Z}^{d} \right\}$$ ha nonempty intersection with $\mathbb{P}-1$ and $\mathbb{P}+1$.
\end{theorem}

For $d=1$, in  \textbf{Section 2}, we will prove a refinement of the Theorem \ref{Prime shift PVW}, which is the following:

\begin{theorem}\label{shift of prime van der}
     Let $A$ be a piecewise syndetic in $\left(\mathbb{N},+\right)$  and $F\in\mathcal{P}_{f}\left(x\mathbb{Z}[x]\right)$, then $$\left\{n\in\mathbb{N}:\bigcap_{f\in F}\left(-f\left(n\right)+A\right)\neq\emptyset\text{ is piecewise syndetic in } \left(\mathbb{N},+\right)\right\}$$ has infinite intersection with $\mathbb{P}-1$ and $\mathbb{P}+1$. 
\end{theorem}

Let $FP\left(\mathbb{P}-1\right)$ be the set of all finite products of distinct elements of  $\mathbb{P}-1$. So $$FP\left(\mathbb{P}-1\right)=\left\{\prod _{x\in H}x: H\in \mathcal{P}_{f}\left(\mathbb{P}-1\right) \right\}.$$  Similarly let $FP\left(\mathbb{P}+1\right)$ be the set of all finite products of distinct elements of  $\mathbb{P}+1$. For $A\subseteq\mathbb{N}$ and $n\in\mathbb{N}$, we define $A/n=\left\{m:mn\in A\right\}$ and $-n+A=\left\{m:m+n\in A\right\}$. In \textbf{Section 2}, we will prove the following:

\begin{theorem}
     Let $r\in\mathbb{N}$ and $\mathbb{N}=\bigcup_{i=}^{r}C_{i}$ and $F\in\mathcal{P}_{f}\left(x\mathbb{Z}[x]\right)$ Then there exists $i\in\left\{1,2,\ldots,r\right\}$  such that  $$\left\{n\in\mathbb{N}:C_{i}/n\cap\bigcap_{f\in F}\left(-f\left(n\right)+C_{i}\right)\neq\emptyset\text{ is piecewise syndetic in } \left(\mathbb{N},+\right)\right\}$$ has infinite intersection with $FP\left(\mathbb{P}-1\right)$. 
\end{theorem}

An analog theorem is also true for $FP\left(\mathbb{P}+1\right)$ and the above theorem is equivalent to the following:

\begin{theorem}\label{shifted prime aboundance moreira}
   Let  $r\in\mathbb{N}$, and let $\mathbb{N}=\bigcup_{i=1}^{r}C_{i}$. There exist $i\in \left\{1,2,\ldots,r\right\}$ and infinitely many $y$ in $FS\left(\mathbb{P}-1\right)$  such that $$\left\{x\in\mathbb{N}:\left\{xy,x+f(y):f\in F\right\}\subseteq C_{i}\right\}$$ is piecewise syndetic, where $F$ is a finite subset of $x\mathbb{Z}\left[x\right]$.
\end{theorem}

An analog version of the above theorem is also true for $FP\left(\mathbb{P}+1\right)$.

\section{The shifted primes Moreira's Theorem}

We start this section, with some definitions. Let $\left(S,+\right)$ be a commutative semigroup, and $A\subseteq S$.
\begin{itemize}

		\item \textbf{($IP_{r}$-set)} Let $r\in \mathbb{N}$. The set $A$ is $IP_{r}$-set if and only if  there exists a sequence $\langle x_{n}\rangle _{n=1}^{r}$ in $S$ such that $FS\left(\langle x_{n}\rangle _{n=1}^{r}\right)\subseteq A$, where $ FS\left(\langle x_{n}\rangle _{n=1}^{r}\right)=\left\{ \sum_{n\in F}x_{n}:F\subseteq\{1,2,\ldots,r\}\right\}$.
		
		\item \textbf{($IP_{r}^{\star}$-set)} Let $r\in \mathbb{N}$. The set $A$ is called $IP_{r}^{\star}$-set, when it intersects with all $IP_{r}$-sets.
		
		\item \textbf{($IP_{0}$-set)} The set $A$ is $IP_{0}$-set if $A$ is $IP_{r}$-set for all $r\in\mathbb{N}$.
		
	\end{itemize}

\begin{theorem}
    For any partition $\mathbb{Z}^{d}=\bigcup_{s=1}^{c}C_{s}$ at least one of the sets $C_{s}$ has the property that for any finite set of polynomials $\vec{f}_{i}:\mathbb{Z}\rightarrow \mathbb{Z}^{d}$, $i=1,\ldots,k$, with $\vec{f}_{i}\left(0\right)=0$ for all $i$, $$\left\{n\in\mathbb{N}:\vec{a},\vec{a}+\vec{f}_{1}\left(n\right),\ldots, \vec{a}+\vec{f}_{k}\left(n\right)\in C_{s}\text{ for some } \vec{a}\in\mathbb{Z}^{d} \right\}$$  is an $IP_{r}^{\star}$-set for $r$ large enough.
\end{theorem}

The authors  of \cite{BLZ}, mentioned that the above theorem can be proved by the  polynomials Hales-Jewett theorem in \cite{BL99}. And by the same theorem, we get the following:

\begin{theorem}\label{pvw ipn}
    For any partition $\mathbb{N}=\bigcup_{s=1}^{c}C_{s}$ at least one of the sets $C_{s}$ has the property that for any finite set of polynomials $f_{i}:\mathbb{Z}\rightarrow \mathbb{Z}$, $i=1,\ldots,k$, with $f_{i}\left(0\right)=0$ for all $i$, $$\left\{n\in\mathbb{N}:a,a + f_{1}\left(n\right),\ldots, a + f_{k}\left(n\right)\in C_{s}\text{ for some } a\in\mathbb{N} \right\}$$  is an $IP_{N}^{\star}$-set for $N$ large enough.
\end{theorem}

 To reach the goal of this article, we need the following refinement of the above theorem, which is proved by S. Goswami in \cite[Theorem 2.3]{G} using the polynomials Hales-Jewett theorem.

\begin{theorem}\label{IPn pol van der ps}
     If $A$ is piecewise syndetic in $\mathbb{N}$  and $F\in\mathcal{P}_{f}\left(x\mathbb{Z}[x]\right)$, then $$\left\{n\in R:\bigcap_{f\in F}\left(-f\left(n\right)+A\right)\neq\emptyset\text{ is piecewise syndetic in } \left(\mathbb{N},+\right)\right\}$$ is an $IP_{N}^{\star}$-set in $\left(\mathbb{N},+\right)$ for  $N$ large enough.
\end{theorem}

\begin{proof}[\textbf{Proof of Theorem \ref{shift of prime van der}}]
    The proof follows from  Theorem \ref{IPn pol van der ps} with the fact that $\mathbb{P}-1$  and $\mathbb{P}+1$ are   $IP_{0}$-sets by \cite{BLZ}.
\end{proof}

As \cite[Theorem 1.10]{HS24} by N. Hindman and D. Strauss, we get the following by using Theorem \ref{shift of prime van der}.

\begin{theorem}\label{main theorem}
     Let $r\in\mathbb{N}$, and let $\mathbb{N}=\bigcup_{i=1}^{r}C_{i}$. There exist $i\in\left\{1,2,\ldots,r\right\}$ a strictly increasing sequence $\langle z_{n}\rangle_{n=1}^{\infty}\subseteq FP\left(\mathbb{P}-1\right) $ in $\mathbb{N}$, and a sequence $\langle E_{n}\rangle_{n=1}^{\infty}$ of piecewise syndetic subsets of $\mathbb{N}$ such that for each $n\in\mathbb{N}$, $E_{n}\subseteq \mathbb{N}z_{n}$ and if $w\in E_{n}$ and $x=wz_{n}^{-1}$, then $\left\{xz_{n},x+f\left(z_{n}\right):f\in F\right\}\subseteq C_{i}$, where $F$ is a finite subset of $x\mathbb{Z}\left[x\right]$.
\end{theorem}

\begin{proof}
    Choose $t_{0}\in\left\{1,2,\ldots,r\right\}$ such that $C_{t_{0}}$ is piecewise syndetic in $\mathbb{N}$ and  pick $y_{1}\in \mathbb{P}-1$,  by Theorem \ref{shift of prime van der}  such that $\bigcap_{f\in F}\left(B_{0}-f\left(y_{1}\right)\right)$ is piecewise syndetic and let $D_{1}=\bigcap_{f\in F}\left(B_{0}-f\left(y_{1}\right)\right)$.
    By  \cite[Lemma 2.5]{HS24}
    $y_{1}D_{1}$ is piecewise syndetic. Since $y_{1}D_{1}=\bigcup_{i=1}^{r}\left(y_{1}D_{1}\cap C_{i}\right)$, pick $t_{1}\in\left\{1,2,\ldots,r\right\}$ such that $y_{1}D_{1}\cap C_{t_{1}}$ is piecewise syndetic and let $B_{1}=\left(y_{1}D_{1}\cap C_{t_{1}}\right)$.

    Let $k\in\mathbb{N}$ and assume we have chosen $\langle y_{j}\rangle_{j=1}^{k}$, $\langle B_{j}\rangle_{j=0}^{k}$, $\langle t_{j}\rangle_{j=0}^{k}$, and $\langle D_{j}\rangle_{j=1}^{k}$ satisfying the following induction hypothesis.

    \begin{itemize}
        \item[(1)] For $j\in\left\{1,2,\ldots,k\right\}$, $y_{j}\in \mathbb{P}-1$ and  $y_{j}>y_{j-1}$.
        \item[(2)] For $j\in\left\{1,2,\ldots,k\right\}$, $D_{j}$ is a piecewise syndetic subset of $\mathbb{N}$.
        \item[(3)] For $j\in\left\{1,2,\ldots,k\right\}$, $t_{j}\in\left\{1,2,\ldots,r\right\}$.
        \item[(4)] For $j\in\left\{1,2,\ldots,k\right\}$, $B_{j}$ is a piecewise syndetic subset of $\mathbb{N}$.
        \item[(5)] For $j\in\left\{1,2,\ldots,k\right\}$, $B_{j}\subseteq C_{t_{j}}$.
        \item[(6)] For $j\in\left\{1,2,\ldots,k\right\}$, $B_{j}\subseteq y_{j}D_{j}$.
        \item[(7)] For $j<m$ in $\left\{0,1,\ldots,k\right\}$, $B_{m}\subseteq y_{m}y_{m-1}\cdots y_{j+1}B_{j}$.
        \item[(8)] For $m\in\left\{1,2,\ldots,k\right\}$, $D_{m}\subseteq B_{m-1}\cap\left(B_{m-1}-y_{m}\right)$ and, if $m>1$, then $$D_{m}\subseteq\bigcap_{j=1}^{m-1}\bigcap_{f\in F}\left(B_{m-1}-y_{m-1}y_{m-2}\ldots y_{j}f\left(y_{m-1}y_{m-2}\ldots y_{j}y_{m}\right)\right)$$.
    \end{itemize}

    All hypotheses hold for $k=1$.

    For $j\in\left\{1,2,\ldots,k\right\}$, let $u_{j}=y_{k}y_{k-1}\ldots y_{j}$ by Theorem \ref{shift of prime van der} , $$A=\left\{y\in S:\bigcap_{j=1}^{k}\bigcap_{f\in F}\left(B_{k}-u_{j}f\left(u_{j}y\right)\right) \text{ is piecewise syndetic }\right\}$$ has infinite intersection  with $\mathbb{P}-1$.
     Pick $y_{k+1}\in A$ with $y_{k+1}>y_{k}$  and $y_{k+1}\in\mathbb{P}-1$ by Theorem \ref{shift of prime van der}. 
     Let $$D_{k+1}=\bigcap_{j=1}^{k}\bigcap_{f\in F}\left(B_{k}-u_{j}f\left(u_{j}y_{k+1}\right)\right).$$  Note that hypotheses (1), (2), and (8) hold at $k+1$.

     By \cite[Lemma 2.5]{HS24} $y_{k+1}D_{k+1}$ is piecewise syndetic. Since $y_{k+1}D_{k+1}=\bigcup_{i=1}^{r}\left(y_{k+1}D_{k+1}\cap C_{i}\right)$, pick $t_{k+1}\in\left\{1,2,\ldots,r\right\}$ such that $y_{k+1}D_{k+1}\cap C_{t_{k+1}}$ is piecewise syndetic and let $B_{k+1}=\left(y_{k+1}D_{1}\cap C_{t_{k+1}}\right)$. Note that hypotheses (3), (4), (5), and (6) hold for $k+1$. We need to verify hypothesis (7) so let $j<m$ in $\left\{0,1,\ldots,k+1\right\}$ be given. If $m\leq k$, then (7) holds by assumption so assume that $m=k+1$. We have $B_{k}\subseteq y_{k}y_{k-1}\cdots y_{j+1}B_{j}$ so $B_{k+1}\subseteq y_{k+1}y_{k}\cdots y_{j+1}B_{j}$ as required.

     The construction is complete. Pick $i\in\left\{1,2,\ldots,r\right\}$ such that $G=\left\{k\in\mathbb{N}:t_{k}=i\right\}$ is infinite. We then choose a sequence $\langle k\left(n\right)\rangle_{n=0}^{\infty}$ in $G$, so that,
     letting $z_{n}=y_{k\left(n\right)}y_{k\left(n\right)-1}\cdots y_{k\left(n-1\right)+1}\in FP\left(\mathbb{P}-1\right)$ for $n\in\mathbb{N}$. For $n\in\mathbb{N}$, let $E_{n}=B_k\left(n\right)$. Then each $E_{n}$ is piecewise syndetic. Also, $$E_{n}=B_k\left(n\right)\subseteq y_{k\left(n\right)}y_{k\left(n\right)-1}\cdots y_{k\left(n-1\right)+1}B_{k\left(n-1\right)}\subseteq z_{n}\mathbb{N}.$$

     Let $w\in E_{n}$ and let $xz_{n}=w$. It is obvious that $w\in C_{i}$. We need to show that $\left\{xz_{n},x+f\left(z_{n}\right):f\in F\right\}\subseteq C_{i}$. It is remains to show that $x+f\left(z_{n}\right)\in C_{i}$  for all $f\in F$. Now

$\begin{aligned}
	z_{n}\left(x+f\left(z_{n}\right)\right) & =  w+z_{n}f\left(z_{n}\right)\\
	& \in B_{k\left(n\right)}+z_{n}f\left(z_{n}\right)\\
 & \subseteq y_{k\left(n\right)}D_{k\left(n\right)}+z_{n}f\left(z_{n}\right)\\
 & \subseteq y_{k\left(n\right)}\left(B_{k\left(n\right)-1}-y_{k\left(n\right)-1}\cdots y_{k\left(n-1\right)+1}f\left(y_{k\left(n\right)}y_{k\left(n\right)-1}\cdots y_{k\left(n-1\right)+1}\right)\right)+z_{n}f\left(z_{n}\right)\\
 & \subseteq y_{k\left(n\right)}B_{k\left(n\right)-1}-z_{n}f\left(z_{n}\right)+z_{n}f\left(z_{n}\right)\\
 & \subseteq z_{n}B_{k\left(n-1\right)}.
\end{aligned}$

  So $x+f\left(z_{n}\right)\in B_{k\left(n-1\right)}\subseteq C_{t_{k(n-1)}}=C_{i}$.   

\end{proof}

\begin{remark}\label{remark 1}
    In the above theorem, we may consider $FP\left(\mathbb{P}+1\right)$ instated of $FP\left(\mathbb{P}-1\right)$.
\end{remark}

\begin{proof}[\textbf{Proof of  Theorem \ref{shifted prime aboundance moreira}}.]
    Pick $i$, $\langle z_{n}\rangle$ and $\langle E_{n}\rangle$ as guaranteed by Theorem \ref{main theorem}. Given $n\in\mathbb{N}$, if $y=z_{n}\in FP\left(\mathbb{P}-1\right)$, then $$E_{n}y^{-1}\subseteq \left\{x\in\mathbb{N}:\left\{xy,x+f(y):f\in F\right\}\subseteq C_{i}\right\}$$ and by  \cite[Lemma 2.7]{HS24} $E_{n}y^{-1}$ is piecewise syndetic.
\end{proof}

\bibliographystyle{plain}

\end{document}